\documentclass[11pt]{article}
\usepackage[left=2.5cm,right=2.5cm,top=2.5cm,bottom=2.5cm,a4paper]{geometry}

\usepackage[a4paper]{geometry}
\usepackage{graphicx}
\usepackage{microtype}
\usepackage{siunitx}
\usepackage{booktabs}
\usepackage{graphics}
\usepackage{graphicx}
\usepackage{epsfig}
\usepackage{amsmath,amsfonts,amssymb,amsthm}
\usepackage{cleveref}
\usepackage{listings}
\usepackage{paralist}
\usepackage{sectsty}
\usepackage{datetime}
\numberwithin{equation}{section}
\subsectionfont{\normalfont}

\newtheorem{theorem}{Theorem}[section]

\newtheorem{lemma}[theorem]{Lemma}
\newtheorem{corollary}[theorem]{Corollary}
\newtheorem{remark}[theorem]{Remark}
\newtheorem{example}[theorem]{Example}

\newtheorem{open problem}[theorem]{Open problem}

\providecommand{\keywords}[1]{\textbf{\textit{Key words---}} #1}

\begin{document}

\markboth{WonTae Hwang and Kyunghwan Song}
{$0<s<1$ rational $s$ }

\title{On the integer part of the reciprocal of the Riemann zeta function tail at certain rational numbers in the critical strip}

\author{WonTae Hwang and Kyunghwan Song\\School of Mathematics, Korea Institute for Advanced Study, Seoul, South Korea \\
Institute of Mathematical Sciences, Ewha Womans University, Seoul,  South Korea}

\maketitle

\begin{abstract}
We prove that the integer part of the reciprocal of the tail of $\zeta(s)$ at a rational number $s=\frac{1}{p}$ for any integer with $p \geq 5$ or $s=\frac{2}{p}$ for any odd integer with $p \geq 5$ can be described essentially as the integer part of an explicit quantity corresponding to it. To deal with the case when $s=\frac{2}{p},$ we use a result on the finiteness of integral points of certain curves over $\mathbb{Q}$. \\

\noindent
\end{abstract}
\section{Introduction}

\qquad Among various kind of zeta functions in mathematics, one of the most famous and important zeta functions is the Riemann zeta function. For $s= \sigma + i t \in \mathbb{C}$ with $\sigma >1,$ consider the absolutely convergent infinite series $\displaystyle \sum_{n=1}^{\infty} \frac{1}{n^s}.$ It is well-known that this series admits an analytic continuation $\zeta(s)$ to the whole complex plane $\mathbb{C}$. If we restrict our attention to some rational numbers $0<s<1,$ then we have the following list of values of the Riemann zeta function:    
\begin{equation*}
\zeta \left(\frac{1}{2} \right) =-1.46035 \cdots,  \zeta \left(\frac{1}{3} \right) = -0.97336 \cdots, \zeta \left(\frac{1}{4} \right) =-0.813278 \cdots,    
\end{equation*}
\begin{equation*}
\zeta \left(\frac{1}{5} \right) =-0.733921 \cdots, \zeta \left(\frac{2}{3} \right) =-2.44758 \cdots, \zeta \left(\frac{2}{5} \right) =-1.1348 \cdots. 
\end{equation*}

Now, for an integer $n \geq 1$ and a real number $s$ with $0<s<1,$ we let
\begin{align*}
\zeta_n(s) &= \frac{1}{1-2^{1-s}} \cdot \sum_{k=n}^{\infty} \frac{(-1)^{k+1}}{k^s}, \\
A_{n,s} & = \left(\frac{1}{n^s} - \frac{1}{(n+1)^s}\right) + \left(\frac{1}{(n+2)^s} - \frac{1}{(n+3)^s}\right) + \cdots,  \\
\intertext{and}
B_{n,s} & = \left(-\frac{1}{n^s} + \frac{1}{(n+1)^s}\right) + \left(-\frac{1}{(n+2)^s} + \frac{1}{(n+3)^s}\right) + \cdots .
\end{align*}

Note that we have 
\begin{equation*}
\zeta_1(s)=\zeta(s)
\end{equation*}
and
\begin{equation}\label{eq:zetaAB}
\zeta_n(s) =
\begin{cases}
- \frac{1}{1-2^{1-s}} A_{n,s}, ~&~\mbox{if~$n$~is even}, \\
- \frac{1}{1-2^{1-s}} B_{n,s}, & ~\mbox{if~$n$~is odd}.
\end{cases}
\end{equation}

Along this line, D. Kim and K. Song \cite{Song2018} (resp. K. Song \cite{Song2019}) proved that we have 
\begin{equation*}
\left[\frac{1}{1-2^{1-s}} \cdot \zeta_n (s)^{-1} \right] = \left[(-1)^{n+1} \cdot 2 \left(n- \frac{1}{2} \right)^s \right]
\end{equation*}
for every integer $n \geq 1,$ at $s=\frac{1}{2}, \frac{1}{3}, \frac{1}{4}$ (resp. $s=\frac{2}{3}).$ In this paper, we extend the previous results to the case when either $s=\frac{1}{p}$ for any integer with $p \geq 5$ or $s=\frac{2}{p}$ for any odd integer with $p \geq 5.$

Our main result is summarized in the following
\begin{theorem}\label{main thm}
 Let $s=\frac{1}{p}$ for any integer with $p \geq 5$ or $s=\frac{2}{p}$ for any odd integer with $p \geq 5.$ Then there exists an integer $N >0$ such that we have
\begin{equation*}
\left[\frac{1}{1-2^{1-s}} \cdot \zeta_n (s)^{-1} \right] = \left[(-1)^{n+1} \cdot 2 \left(n- \frac{1}{2} \right)^s \right]
\end{equation*}
for every integer $n \geq N.$ 
\end{theorem}

For more details, see Corollaries \ref{main cor 1} and \ref{main cor 2}. \\ 

This paper is organized as follows: In Section \ref{pre}, we first introduce some properties of $\zeta_n (s)$. Afterwards, we recall a theorem of Siegel on the integral points of a smooth algebraic curve over a number field (see Theorem \ref{thm_Siegel}). In Section \ref{1/p}, we deal with the case of $s=\frac{1}{p}$, using Theorem \ref{main tool2} below. In Section \ref{2/p}, we give a proof of Theorem \ref{main thm} for the case when $s=\frac{2}{p}$ by invoking a version of the previously introduced theorem of Siegel (see Theorem \ref{thm_hyper}).  

\section{Preliminaries}\label{pre}

\subsection{Properties of $\zeta_n (s)$}
\qquad In this section, we give some useful properties of $\zeta_n (s)$ in terms of the size of its reciprocal. To achieve our goal, we first need the following 

\begin{theorem}\label{main tool1}
Let $s$ be a real number with $0<s<1$. Then we have
\begin{equation*}
2\left(n-\frac{1}{2}\right)^{s} < A^{-1}_{n,s} < 2\left(n-\frac{1}{4}\right)^{s}
\end{equation*}
for every even integer $n \geq 2$, and 
\begin{equation*}
-2\left(n-\frac{1}{4}\right)^{s} < B^{-1}_{n,s} < -2\left(n-\frac{1}{2}\right)^{s} .
\end{equation*}
for every odd integer $n \geq 1$.
\end{theorem}
\begin{proof}
For a proof, see \cite[Theorem 1]{Song2018}.
\end{proof}
In view of the equation (\ref{eq:zetaAB}), Theorem \ref{main tool1} has a nice consequence:

\begin{corollary}
For any real number $s$ with $0<s<1$, we have
\begin{gather*}
-2(1-2^{1-s}) \left(n-\frac{1}{2}\right)^s < \zeta_n(s)^{-1} < -2(1-2^{1-s})\left(n-\frac{1}{4}\right)^s
\end{gather*}
for every even integer $n \geq 2$, and
\begin{gather*}
2(1-2^{1-s}) \left(n-\frac{1}{4}\right)^s < \zeta_n(s)^{-1} < 2(1-2^{1-s})\left(n-\frac{1}{2}\right)^s
\end{gather*}
for every odd integer $n \geq 1$.
\end{corollary}

If we do not require the inequality of Theorem \ref{main tool1} to hold for every even or odd integer, as indicated above, then we can obtain a slightly better upper bound of the value of $\zeta_n (s)^{-1}$:
\begin{theorem}\label{main tool2}
	Let $\epsilon>0$ be given. Then for any real number $s$ with $0<s<1$, we have
	\begin{gather*}
	2\left(n-\frac{1}{2}\right)^s < A^{-1}_{n,s} < 2\left(n - \frac{1}{2} + \epsilon\right)^s
	\end{gather*}
	for every sufficiently large even integer $n$, and
	\begin{gather*}
	-2\left(n-\frac{1}{2} + \epsilon \right)^s < B^{-1}_{n,s} \leq -2\left(n - \frac{1}{2} \right)^s
	\end{gather*}
	for every sufficiently large odd integer $n$.
\end{theorem}
\begin{proof}
For a proof, see \cite[Theorem 2]{Song2018}.
\end{proof}

As before, combining the equation (\ref{eq:zetaAB}) and Theorem \ref{main tool2} yields the following
\begin{corollary}
	Let $\epsilon>0$ be given. Then for any real number $s$ with $0<s<1$, we have
        \begin{gather*}
	-2(1-2^{1-s}) \left(n-\frac{1}{2}\right)^s < \zeta_n(s)^{-1} < -2(1-2^{1-s})\left(n-\frac{1}{2} + \epsilon \right)^s
	\end{gather*}
	for every sufficiently large even integer $n$, and
	\begin{gather*}
	2(1-2^{1-s}) \left(n-\frac{1}{2} + \epsilon \right)^s < \zeta_n(s)^{-1} < 2(1-2^{1-s})\left(n-\frac{1}{2}\right)^s
	\end{gather*}
	for every sufficiently large odd integer $n$.
\end{corollary}

\subsection{Siegel's theorem on integral points}
\qquad In this section, we briefly review a theorem of Siegel on the integral points of certain curves that are defined over a number field. \\

In the sequel, let $K$ be a number field, $S$ a finite set of places of $K,$ and $R_S$ the ring of $S$-integers in $K.$ Also, let $\overline{K}$ be an algebraic closure of $K.$ Then we have the following fundamental result:  
 
\begin{theorem}\label{thm_Siegel}
Let $C$ be a smooth projective curve of genus $g$ over $K,$ and let $f \in K(C)$ be a nonconstant function. If $g \geq 1,$ then the set 
\begin{equation*}
\{P \in C(K)~|~f(P) \in R_S \}
\end{equation*}
is finite.
\end{theorem}
\begin{proof}
For a proof, see \cite[Theorem D.9.1]{3}.
\end{proof}
In fact, this theorem is more general than what we actually need. We need to use a version of Theorem \ref{thm_Siegel} regarding a hyperelliptic curve, which we describe now: suppose that $S$ includes all the infinite places. 
\begin{theorem}\label{thm_hyper}
Let $f(X) \in K[X]$ be a polynomial of degree at least $3$ with distinct roots in $\overline{K}$. Then the equation $Y^2 = f(X)$ has only finitely many solutions $X,Y \in R_S.$
\end{theorem}
\begin{proof}
For a proof, see \cite[Theorem D.8.3]{3}.
\end{proof}

\begin{example}\label{main exam}
Let $K=\mathbb{Q}$, $S=\{2, \infty\}$, and let $f(X)=\frac{1}{32}X^5 + \frac{17}{32}.$ Then the equation $Y^2 = f(X)$ has only finitely many solutions in $R_S= \mathbb{Z} \left[\frac{1}{2} \right].$ This example is closely related to the case of $s=\frac{2}{5}$ and $m=1$ in Lemma \ref{Siegel lem} below.
\end{example}

\section{The case of $s = \frac{1}{p}$}\label{1/p}
Throughout this section, let $p \geq 5$ be a fixed integer and let $s=\frac{1}{p}$. 
\begin{lemma}\label{no int lem}
Let $n \geq 1$ be an integer. Then there is no integer between $ 2 \left(n - \frac{1}{2} \right)^s$ and $ 2 \left(n - \frac{1}{2} +\frac{1}{2^{p+1}} \right)^s.$
\end{lemma}
\begin{proof}
Suppose on the contrary that there is an integer $x$ with $2 \left(n - \frac{1}{2} \right)^s < x < 2 \left(n - \frac{1}{2} +\frac{1}{2^{p+1}} \right)^s.$ It follows that 
\begin{equation*}
2^p n - 2^{p-1} < x^p < 2^p n - 2^{p-1} +\frac{1}{2},
\end{equation*}
which is absurd. \\

This completes the proof.
\end{proof}
By a similar argument, we can also have the following
\begin{lemma}\label{no int lem1}
Let $n \geq 1$ be an integer. Then there is no integer between $ -2 \left(n - \frac{1}{2} + \frac{1}{2^{p+1}} \right)^s$ and $ -2 \left(n - \frac{1}{2} \right)^s.$
\end{lemma}

Combining all these results, we have:
\begin{corollary}\label{main res1}
There exist integers $n_0, n_1>0$ such that 
\begin{gather*}
[A_{n,s}^{-1}] = \left[ 2\left( n-\frac{1}{2} \right)^s \right]
\end{gather*}
for every even integer $n \geq n_0$, and 
\begin{gather*}
[B_{n,s}^{-1}] = \left[ -2\left( n-\frac{1}{2} \right)^s \right]
\end{gather*}
for every odd integer $n \geq n_1.$
\end{corollary}
\begin{proof}
Let $\epsilon = \frac{1}{2^{p+1}}.$ By Theorem \ref{main tool2}, there exist integers $n_0, n_1>0$ such that 
\begin{equation*}
2 \left(n - \frac{1}{2} \right)^s < A_{n,s}^{-1} < 2 \left(n - \frac{1}{2} +\frac{1}{2^{p+1}} \right)^s
\end{equation*}
for every even integer $n \geq n_0,$ and 
\begin{equation*}
-2 \left(n - \frac{1}{2}+\frac{1}{2^{p+1}} \right)^s < B_{n,s}^{-1} < -2 \left(n - \frac{1}{2} \right)^s
\end{equation*}
for every odd integer $n \geq n_1.$ Then by Lemmas \ref{no int lem} and \ref{no int lem1}, it follows that $[A_{n,s}^{-1}] = \left[ 2\left( n-\frac{1}{2} \right)^s \right]$ for every even integer $n \geq n_0,$ and $[B_{n,s}^{-1}] = \left[ -2\left( n-\frac{1}{2} \right)^s \right]$ for every odd integer $n \geq n_1,$ as desired. \\

This completes the proof.
\end{proof}
An immediate consequence of Corollary \ref{main res1} is the following
\begin{corollary}\label{main cor 1}
There exists an integer $N >0$ such that we have
\begin{equation*}
\left[\frac{1}{1-2^{1-s}} \cdot \zeta_n (s)^{-1} \right] = \left[ (-1)^{n+1} \cdot 2 \left(n - \frac{1}{2} \right)^s \right]
\end{equation*}
for every integer $n \geq N.$
\end{corollary}
\begin{proof}
This follows from the equation (\ref{eq:zetaAB}) and Corollary \ref{main res1}.
\end{proof}

\begin{remark}\label{rem 1}
If $p \in \{2,3,4 \},$ then we may take $N=1$ in view of \cite[Corollary 6]{Song2018}.
\end{remark}

\section{The case of $s = \frac{2}{p}$}\label{2/p}
Throughout this section, let $p \geq 5$ be a fixed odd integer and let $s=\frac{2}{p}$. \\

To begin with, we introduce one useful inequality:
\begin{lemma}\label{lem2/3_0}
For any real number $x \geq 2,$ we have
\begin{equation*}
\left(\left(1-\frac{1}{4x}\right)^2 + \frac{1}{8x^2}\right)^{-1/p-1} \cdot \left(1-\frac{1}{4x}\right) + \left(\left(1+\frac{1}{4x}\right)^2 + \frac{1}{8x^2}\right)^{-1/p-1} \cdot \left(1+\frac{1}{4x}\right) < 2.
\end{equation*}
\end{lemma}
\begin{proof}
For convenience, let $\alpha=4x$ and $h(\alpha) = \left(\left(1-\frac{1}{\alpha}\right)^2 + \frac{2}{\alpha^2}\right)^{-1/p-1} \cdot \left(1-\frac{1}{\alpha}\right)$. (Note that $\alpha \geq 8$.) Then it suffices to show that $h(\alpha) + h(-\alpha)$ is a strictly increasing function because we have
	\begin{gather*}
	\lim_{\alpha \rightarrow \infty} \left(\left(\left(1-\frac{1}{\alpha}\right)^2 + \frac{2}{\alpha^2}\right)^{-1/p-1} \cdot \left(1-\frac{1}{\alpha}\right) + \left(\left(1+\frac{1}{\alpha}\right)^2 + \frac{2}{\alpha^2}\right)^{-1/p-1} \cdot \left(1+\frac{1}{\alpha}\right)\right) = 2.
	\end{gather*} 
	Indeed, we will use the first derivative test, as follows: note first that we have
	\begin{align*}
	& \frac{d\left(h(\alpha) + h(-\alpha)\right)}{d \alpha} \\
	& = \frac{d \left(\left(\left(1-\frac{1}{\alpha}\right)^2 + \frac{2}{\alpha^2}\right)^{-1/p-1} \cdot \left(1-\frac{1}{\alpha}\right) + \left(\left(1+\frac{1}{\alpha}\right)^2 + \frac{2}{\alpha^2}\right)^{-1/p-1} \cdot \left(1+\frac{1}{\alpha}\right)\right)}{d \alpha} \\
	& = \frac{1}{\alpha^4 p}\Bigg( \alpha^2 p\left( \left( \frac{\alpha^2 - 2\alpha + 3}{\alpha^2} \right)^{-1-1/p} - \left(\frac{\alpha^2 + 2\alpha + 3}{\alpha^2} \right)^{-1-1/p}   \right) \\
	& \quad - 2(p+1) \left(  (\alpha - 1)(\alpha - 3) \left( \frac{\alpha^2 - 2\alpha + 3}{\alpha^2} \right)^{-2-1/p} -(\alpha +1)(\alpha +3) \left( \frac{\alpha^2 + 2\alpha + 3}{\alpha^2} \right)^{-2-1/p}   \right)\Bigg).
	\end{align*}
	Since $\alpha^4 p >0$, we need to show that
	\begin{align*}
	& \alpha^2 p\left( \left( \frac{\alpha^2 - 2\alpha + 3}{\alpha^2} \right)^{-1-1/p} - \left(\frac{\alpha^2 + 2\alpha + 3}{\alpha^2} \right)^{-1-1/p}   \right) \\
	& \quad - 2(p+1) \left(  (\alpha - 1)(\alpha - 3) \left( \frac{\alpha^2 - 2\alpha + 3}{\alpha^2} \right)^{-2-1/p} -(\alpha +1)(\alpha +3) \left( \frac{\alpha^2 + 2\alpha + 3}{\alpha^2} \right)^{-2-1/p}   \right) > 0, 
\end{align*}
or equivalently, (by multiplying $\alpha^{-4-2/p}$ and rearranging),
\begin{align*}
& (\alpha^2 - 2\alpha + 3)^{-2-1/p} \left( p (\alpha^2 - 2\alpha + 3) - 2(p+1)(\alpha - 1)(\alpha - 3)  \right) \\
& > (\alpha^2 + 2\alpha + 3)^{-2-1/p} \left( p (\alpha^2 + 2\alpha + 3) - 2(p+1) (\alpha +1)(\alpha +3)\right), 
\end{align*}
which is equivalent to saying that
\begin{equation*}
\frac{(-p-2)\alpha^2 + (6p + 8)\alpha + (-3p-6)}{(-p-2)\alpha^2 + (-6p-8)\alpha + (-3p-6)} < \left( \frac{\alpha^2 + 2\alpha + 3}{\alpha^2 - 2\alpha + 3} \right)^{-2-1/p}, 
\end{equation*}
because $(-p-2)\alpha^2 + (-6p-8)\alpha + (-3p-6) < 0$ for any $\alpha, p$. We prove the last inequality.
	Indeed, we may assume that $(-p-2)\alpha^2 + (6p + 8)\alpha + (-3p-6) < 0$ because if $(-p-2)\alpha^2 + (6p + 8)\alpha + (-3p-6) \geq 0$, then the desired inequality follows trivially. Hence, we have to show that
	\begin{equation}\label{eqn 1}
	\left( \frac{\alpha^2 + 2\alpha + 3}{\alpha^2 - 2\alpha + 3} \right)^{2+1/p} < \frac{(p+2)\alpha^2 + (6p + 8)\alpha + (3p+6)}{(p+2)\alpha^2 + (-6p-8)\alpha + (3p+6)}.
	\end{equation}
	Since
	\begin{equation*}
	\left( \frac{\alpha^2 + 2\alpha + 3}{\alpha^2 - 2\alpha + 3} \right)^{2+1/p} < \left( \frac{\alpha^2 + 2\alpha + 3}{\alpha^2 - 2\alpha + 3} \right)^{7/3}
	\end{equation*}
	and
	\begin{align*}
	& (\alpha^2 + 2\alpha  +3)^{7} \left( (p+2)\alpha^2 + (-6p-8)\alpha + (3p+6)  \right)^3 \\
	& \quad - (\alpha^2 - 2\alpha  +3)^{7} \left( (p+2)\alpha^2 + (6p+8)\alpha + (3p+6)\right)^3  \\
	& = (-8p^3 - 24p^2 + 32)\alpha^{19} + (-88p^3 - 1032p^2 - 2304p -1056)\alpha^{17} \\
	& \quad + (5472p^3 + 21792p^2 + 16128p + 2304)\alpha^{15} + (-25952p^3 + 62688p^2 + 126720p + 52224)\alpha^{13} \\
	& \quad + (-656496p^3-1239120p^2-856320p-217024)\alpha^{11} \\
	& \quad + (-1969488p^3 - 3717360p^2 -2568960p -651072)\alpha^{9} \\
	& \quad + (-700704p^3 + 1692576p^2 + 3421440p + 1410048)\alpha^7 \\
	& \quad + (1329696p^3 + 5295456p^2 + 3919104p + 559872)\alpha^5 \\
	& \quad + (-192456p^3 -2256984p^2 -5038848p -2309472)\alpha^3 + (-157464p^3 - 472392p^2 + 629856)\alpha \\
	& < (-8\alpha^4 + 5472)p^3 \alpha^{15} + (-20\alpha^{4} + 21792)p^2 \alpha^{15} + (-4p^2\alpha^6 +32\alpha^{6} + 52224)\alpha^{13}  \\
	& \quad +(-2304\alpha^{2} + 16128)p \alpha^{15} + (-1056\alpha^2 + 2304)\alpha^{15} + (-1032p\alpha^{4} + 62688p + 126720)p\alpha^{13} \\
	& \quad + (-656496\alpha^6 + 1329696)p^3\alpha^{5} + (-1239120\alpha^6 + 5295456)p^2 \alpha^5 + (-856320\alpha^6 + 3919104)p\alpha^5 \\
	& \quad + (-217024\alpha^6 + 559872)\alpha^{5} + \left(-700704 + \frac{1692576}{5} + \frac{3421440}{5^2} + \frac{1410048}{5^3}\right)p^3 \alpha^7 \\
	& \quad + (-192456p^3 -2256984p^2 -5038848p -2309472)\alpha^3 + \left(-157464p^3 - 472392p^2 + 629856\right)\alpha \\
	& <0
	\end{align*}
for every $p\geq 5$ and $\alpha \geq 8$, we have that the inequality (\ref{eqn 1}) holds for $\alpha \geq 8$, which in turn, implies that $h(\alpha) + h(-\alpha)$ is increasing for $\alpha \geq 8$. \\

This completes the proof.
\end{proof}

An important consequence of the above lemma is the following 
\begin{lemma} \label{lem_2/3_1}
For any even integer $n \geq 4$, we have
\begin{gather*}
2\left(n-\frac{1}{2}\right)^{s} = 2\left(n^2 -n +\frac{1}{4}\right)^{s/2} < A^{-1}_{n,s} < 2\left(n^2 -n +\frac{3}{4}\right)^{s/2}
\end{gather*}
and for any odd integer $n \geq 3$, we have
\begin{gather}\label{eqn 2}
-2\left(n^2 -n +\frac{3}{4}\right)^{s/2} < B^{-1}_{n,s} < -2\left(n^2 -n +\frac{1}{4}\right)^{s/2}= -2\Big(n-\frac{1}{2}\Big)^{s}.
\end{gather}
\end{lemma}
\begin{proof}
	Let $n \geq 4$ be an even integer. By Theorem \ref{main tool1}, we have
	\begin{gather*}
	2\left( n - \frac{1}{2} \right)^{2/p} < A^{-1}_{n,2/p} < 2\left( n - \frac{1}{4} \right)^{2/p}.
	\end{gather*}
	Hence, it suffices to show that
	\begin{gather*}
	\frac{1}{n^{2/p}} - \frac{1}{(n+1)^{2/p}} - \frac{1}{2} \left( \left(n - \frac{1}{2}\right)^2 + \frac{1}{2} \right)^{-1/p} + \frac{1}{2} \left( \left(n + \frac{3}{2}\right)^2 + \frac{1}{2} \right)^{-1/p}   > 0
	\end{gather*}
	for any even integer $n \geq 4$. Let $f,g: \mathbb{R}_{\geq 2} \rightarrow \mathbb{R}$ be two functions defined by
	\begin{gather*}
	f(x) = \frac{1}{(2x)^{2/p}} - \frac{1}{(2x+1)^{2/p}} - \frac{1}{2} \left( \left(2x - \frac{1}{2}\right)^2 + \frac{1}{2} \right)^{-1/p} + \frac{1}{2} \left( \left(2x + \frac{3}{2}\right)^2 + \frac{1}{2} \right)^{-1/p}
	\end{gather*}
	and 
	\begin{align*}
	g(x) & = (2x)^{-2/p} - \frac{1}{2} \left( \left(2x-\frac{1}{2}\right)^2 + \frac{1}{2} \right)^{-1/p} - \frac{1}{2} \left( \left(2x+\frac{1}{2}\right)^2 + \frac{1}{2} \right)^{-1/p} \\
	& = (2x)^{-2/p} - \frac{1}{2} \left(4x^2 - 2x + \frac{3}{4} \right)^{-1/p} - \frac{1}{2} \left(4x^2 + 2x + \frac{3}{4} \right)^{-1/p}.
	\end{align*}
	Then we have $f(x) = g(x) - g(x+\frac{1}{2})$, and hence, we only need to show that $g(x)$ is decreasing.  Since
	\begin{align*}
	g'(x) & = \left(-\frac{4}{p}\right)\cdot (2x)^{-2/p - 1} \\
	& + \frac{1}{2p} \left(4x^2 - 2x + \frac{3}{4}\right)^{-1/p-1} \cdot (8x-2) + \frac{1}{2p} \left(4x^2 + 2x + \frac{3}{4}\right)^{-1/p-1} \cdot (8x+2),
	\end{align*}
	we have to show that
	\begin{align*}
	& (2x)^{-2/p-1} - \left(4x^2 - 2x + \frac{3}{4}\right)^{-1/p-1} \cdot \left(x-\frac{1}{4}\right) - \left(4x^2 + 2x + \frac{3}{4}\right)^{-1/p-1} \cdot \left(x+\frac{1}{4}\right) > 0, \\
\textrm{which }& \textrm{is equivalent to saying that} \\
	 & \left(\left(1-\frac{1}{4x}\right)^2 + \frac{1}{8x^2}\right)^{-1/p-1} \cdot \left(1-\frac{1}{4x}\right) + \left(\left(1+\frac{1}{4x}\right)^2 + \frac{1}{8x^2}\right)^{-1/p-1} \cdot \left(1+\frac{1}{4x}\right) < 2.
	\end{align*}
Then it follows from Lemma \ref{lem2/3_0} that $g(x)$ is decreasing for $x \geq 2$. Also, a similar argument can be used to show that (\ref{eqn 2}) holds for any odd integer $n \geq 3.$ \\

This completes the proof.
\end{proof}
\begin{remark}
The inequalities in the statement of Lemma \ref{lem_2/3_1} also hold when $n=1,2$. For example, we have
\begin{equation*}
B^{-1}_{1,2/5} + 2\cdot \left(\frac{3}{4}\right)^{1/5} = 0.179457... > 0 ~\textrm{and}~ A^{-1}_{2,2/5} - 2\cdot \left(\frac{11}{4} \right)^{1/5} = -0.0374817... < 0.
\end{equation*}
\end{remark}

Now, we give a result on the finiteness of the integer points of certain affine curves, which will be used later:
\begin{lemma}\label{Siegel lem}
The affine curve $C_m : x^p -2^p y^2 -2^p y +2^{p-2} +m =0$ defined over $\mathbb Q$ has only finitely many integer points for each $1 \leq m \leq 2^{p-1}-1.$ 
\end{lemma}
\begin{proof}
Let $1 \leq m \leq 2^{p-1}-1$ be fixed. By completing the square and rearranging, the defining equation of $C_m$ can be written as
\begin{equation}\label{eqn 3}
\left(y+\frac{1}{2}\right)^2 = \frac{1}{2^p} \cdot x^p + \frac{2^{p-1}+m}{2^p}.
\end{equation} 
By letting $Y=y+\frac{1}{2},$ the equation (\ref{eqn 3}) becomes
\begin{equation}\label{eqn 4}
Y^2 = \frac{1}{2^p} \cdot x^p + \frac{2^{p-1}+m}{2^p}.
\end{equation}
Let $S=\{2, \infty\}$ be a finite set of places of $\mathbb{Q}$, and let $f(x)= \frac{1}{2^p} \cdot x^p + \frac{2^{p-1}+m}{2^p}.$ Then the equation $Y^2 = f(x)$ has only finitely many solutions in $R_S= \mathbb{Z} \left[\frac{1}{2} \right]$ by Theorem \ref{thm_hyper}, which in turn, implies that $C_m (\mathbb{Z})$ is also finite. (Note that if $(a,b)$ is an integer point of $C_m$, then $\left(a, b+\frac{1}{2} \right)$ is a solution of the equation (\ref{eqn 4}) in $R_S=\mathbb{Z} \left[\frac{1}{2}\right].$) Since $m$ was arbitrary, the proof is complete.
\end{proof}

In the sequel, let $C_m$ denote the affine curve defined as in Lemma \ref{Siegel lem} for $1 \leq m \leq 2^{p-1}-1$. Using the above finiteness result on integral points, we have the following
\begin{lemma}\label{no int lem2}
There exists an integer $n_0>0$ with the property that there is no integer between $ 2 \left(n^2 -n + \frac{1}{4} \right)^{\frac{s}{2}}$ and $ 2 \left(n^2 -n+ \frac{3}{4}  \right)^{\frac{s}{2}}$ for every integer $n \geq n_0.$
\end{lemma}
\begin{proof}
Let $S$ be the set of integers $n$ such that there is an integer between $2 \left(n^2 -n + \frac{1}{4} \right)^{\frac{s}{2}}$ and $ 2 \left(n^2 -n+ \frac{3}{4}  \right)^{\frac{s}{2}}.$ Let $n \in S$ and suppose that $a$ is an integer with $2 \left(n^2 -n + \frac{1}{4} \right)^{\frac{s}{2}} < a < 2 \left(n^2 -n+ \frac{3}{4}  \right)^{\frac{s}{2}}.$ (Note that, in view of \cite[page 9]{Song2018}, there is exactly one such $a.$) It follows that 
\begin{equation*}
2^p n^2 - 2^{p}n +2^{p-2} < a^p < 2^p n^2 - 2^{p}n +3 \cdot 2^{p-2}.
\end{equation*}
We may write $a^p = 2^p n^2 -2^p n + 2^{p-2}+l$ for some $1 \leq l \leq 2^{p-1}-1$ so that we have $(a,n) \in C_l(\mathbb{Z}).$ (Note also that the uniqueness of $a$ guarantees the uniqueness of such $l.$) This observation gives rise to a well-defined set map $\displaystyle \varphi : S \rightarrow \bigcup_{m=1}^{2^{p-1}-1} C_m(\mathbb{Z})$ given by $\varphi (n) = (a,n)$, where $a$ is uniquely determined as above. It is easy to see that $\varphi$ is injective, by construction. Now, by Lemma \ref{Siegel lem}, the set $\displaystyle \bigcup_{m=1}^{2^{p-1}-1} C_m(\mathbb{Z})$ is finite, and hence, it follows that $S$ is also finite. Let $n_0 = \max S +1.$ Then for any integer $n \geq n_0$, we have $n \not \in S,$ which is the desired result. \\

This completes the proof.
\end{proof}
By a similar argument, we can also have the following
\begin{lemma}\label{no int lem3}
There exists an integer $n_1>0$ with the property that there is no integer between $ -2 \left(n^2 -n + \frac{3}{4} \right)^{\frac{s}{2}}$ and $ -2 \left(n^2 -n+ \frac{1}{4}  \right)^{\frac{s}{2}}$ for every integer $n \geq n_1.$
\end{lemma}

Using Lemmas \ref{lem_2/3_1}, \ref{no int lem2}, and \ref{no int lem3}, we can prove the following
\begin{theorem}\label{main res2}
There exist integers $n_0, n_1>0$ such that 
\begin{gather*}
[A_{n,s}^{-1}] = \left[ 2\left( n-\frac{1}{2} \right)^s \right]
\end{gather*}
for every even integer $n \geq n_0$, and 
\begin{gather*}
[B_{n,s}^{-1}] = \left[ -2\left( n-\frac{1}{2} \right)^s \right]
\end{gather*}
for every odd integer $n \geq n_1.$
\end{theorem}
\begin{proof}
By Lemmas \ref{lem_2/3_1} and \ref{no int lem2}, there exists an integer $n_0 >0$ such that
\begin{equation*}
2 \left(n - \frac{1}{2} \right)^s < A_{n,s}^{-1} < 2 \left(n^2 - n + \frac{3}{4} \right)^{\frac{s}{2}}
\end{equation*}
and there is no integer between $2 \left(n - \frac{1}{2} \right)^s$ and $2 \left(n^2 - n + \frac{3}{4} \right)^{\frac{s}{2}}$ for every even integer $n \geq n_0.$ Then it follows that $\left[A_{n,s}^{-1}\right]=\left[2 \left(n-\frac{1}{2} \right)^s \right]$ for every even integer $n  \geq n_0.$ Similarly, by Lemmas \ref{lem_2/3_1} and \ref{no int lem3}, there exists an integer $n_1 >0$ such that
\begin{equation*}
-2 \left(n^2 - n + \frac{3}{4} \right)^{\frac{s}{2}} < B_{n,s}^{-1} < -2 \left(n- \frac{1}{2} \right)^{s}
\end{equation*}
and there is no integer between $-2 \left(n^2 - n + \frac{3}{4} \right)^{\frac{s}{2}}$ and $-2 \left(n- \frac{1}{2} \right)^{s}$ for every odd integer $n \geq n_1.$ Then it follows that $\left[B_{n,s}^{-1}\right]=\left[-2 \left(n-\frac{1}{2} \right)^s \right]$ for every odd integer $n  \geq n_1.$  \\

This completes the proof. 
\end{proof}

As an immediate consequence of Theorem \ref{main res2}, we have:
\begin{corollary}\label{main cor 2}
There exists an integer $N >0$ such that we have
\begin{equation*}
\left[\frac{1}{1-2^{1-s}} \cdot \zeta_n (s)^{-1} \right] = \left[ (-1)^{n+1} \cdot 2 \left(n - \frac{1}{2} \right)^s \right]
\end{equation*}
for every integer $n \geq N.$
\end{corollary}
\begin{proof}
This follows from the equation (\ref{eq:zetaAB}) and Theorem \ref{main res2}.
\end{proof}
\begin{remark}
If $p=3,$ then we may take $N=1$ in view of \cite[Theorem 3.12]{Song2019}.
\end{remark}\label{rem 2}
We conclude this paper with the following
\begin{remark}
In light of \cite[Remark D.9.5]{3}, it might be possible to find such a suitable $N >0$ in Corollaries \ref{main cor 1} and \ref{main cor 2} by adopting a theorem of Baker (see \cite{Baker}).    
\end{remark}

\end{document}